\documentclass[12pt]{article}
\usepackage[margin=1in]{geometry}
\usepackage{amsmath,amssymb,amsthm,xcolor}
\usepackage{enumitem}
\usepackage[colorlinks=true,allcolors=blue,linktoc=all]{hyperref}
\newcommand{\Sph}{\mathbb S}
\newcommand{\VR}{\operatorname{VR}}
\newcommand{\coind}{\operatorname{coind}}
\newcommand{\Bor}{\operatorname{Bor}}
\newcommand{\dgh}{d_{\mathrm{GH}}}
\newcommand{\dI}{d_{\mathrm I}}
\newcommand{\CapVol}{\mathcal V}
\newcommand{\dis}{\operatorname{dis}}
\newcommand{\diam}{\operatorname{diam}}
\newcommand{\Av}{\operatorname{av}}
\newtheorem{theorem}{Theorem}[section]
\newtheorem*{maintheorem}{Main Theorem}
\newtheorem{lemma}[theorem]{Lemma}
\newtheorem{proposition}[theorem]{Proposition}
\newtheorem{corollary}[theorem]{Corollary}
\theoremstyle{definition}
\newtheorem{definition}[theorem]{Definition}
\newtheorem{question}{Question}
\theoremstyle{remark}
\newtheorem{remark}[theorem]{Remark}

\setlist{nosep}
\title{Gromov–Hausdorff Distances Between Spheres: Asymptotics and Vietoris–Rips Coindex Bounds}
\author{Sunhyuk Lim\thanks{Department of Mathematics, Sungkyunkwan University
  (SKKU), Suwon, Republic of Korea.\protect\newline Email:
  \href{mailto:lsh3109@skku.edu}{\texttt{lsh3109@skku.edu}}.}
  \and
  Facundo M\'emoli\thanks{Department of Mathematics, Rutgers, The State
  University of New Jersey, Piscataway, NJ, USA.\protect\newline Email:
  \href{mailto:facundo.memoli@gmail.com}{\texttt{facundo.memoli@gmail.com}}.}
  \and
  Qingsong Wang\thanks{Hal{\i}c{\i}o\u{g}lu Data Science Institute,
  University of California San Diego, La Jolla, CA, USA.\protect\newline Email:
  \href{mailto:qiw072@ucsd.edu}{\texttt{qiw072@ucsd.edu}}.}}
\hypersetup{pdfauthor={Sunhyuk Lim, Facundo M\'emoli, Qingsong Wang}}
\date{}

\begin{document}
\maketitle

\begin{abstract}
For \(1\le m<n\), let \(c_{m,n}\) be the infimum of the scales
\(r\) at which the Vietoris--Rips simplicial complex of the unit round
sphere \(\Sph^m\), equipped with its geodesic metric, admits a continuous
odd map from \(\Sph^n\).
Previous work established the lower bound
\[
 \tfrac12c_{m,n}\le\dgh(\Sph^m,\Sph^n)
\]
through a quantitative Borsuk--Ulam theorem and raised the question of
whether it is always sharp. We show that this bound can underestimate the
distance by a factor arbitrarily close to two, the largest possible factor.
Combining projective-code upper bounds for
\(c_{m,n}\) with lower bounds for \(\dgh\) obtained from obstructions to
homomorphisms between Borsuk graphs, we obtain explicit finite-dimensional
bounds. Along sequences of pairs \(1\le m<n\), we have
\(c_{m,n}\to\tfrac\pi2\) whenever \(\log n/m\to0\), while
\(\dgh(\Sph^m,\Sph^n)\to\tfrac\pi2\) whenever \(m/n\to0\).
When both conditions hold (for example when $n=m^2$) the ratio of the distance to
\(\tfrac12c_{m,n}\) tends to two. For each fixed \(m\ge3\), the inequality
is strict for all sufficiently large \(n\).
\end{abstract}

\noindent\textbf{2020 Mathematics Subject Classification.}
Primary 53C23; secondary 51F30, 55M20, 05C60.

\smallskip
\noindent\textbf{Keywords.}
Gromov--Hausdorff distance, round spheres, Vietoris--Rips complexes,
equivariant coindex, Borsuk graphs, projective codes, spherical isodiametry.
\par

\section{Introduction}
\label{sec:stream-introduction}

The Gromov--Hausdorff distance, introduced by Edwards
\cite{edwards1975structure} and developed systematically by Gromov
\cite{gromov1999metric}, is a fundamental notion in metric geometry.
It provides a framework for studying convergence and compactness of families
of Riemannian manifolds, whose limits need not themselves be manifolds
\cite{gromov1999metric,burago2022course}.

The Gromov--Hausdorff distance provides a mathematical model for
comparing shapes and metric data by seeking matchings that approximately
preserve pairwise distances \cite{memoli-sapiro,memoli-properties}.
The related Gromov--Wasserstein framework  \cite{memoli-2011-GW,sturm-GW} extends this comparison to
metric measure spaces and underlies computational methods for shape and
data matching \cite{peyre2016gromov}.
The Gromov--Hausdorff distance also controls the stability of persistent
invariants \cite{carlson-memoli-clustering,chazal2009gromov,blumberg2017universality,lim2020vietoris}. In the case of finite metric spaces, its computation is NP-hard \cite{schmiedl2017computational}. Exact values and rigorous
bounds for standard families consequently provide useful \emph{benchmarks} for
approximation methods and lower-bound invariants. Round spheres furnish
such a family, on which the accuracy of these methods can be assessed
against geometric estimates for the distance.

The problem of determining the
Gromov--Hausdorff distance between unit round spheres \((\Sph^m,d_m)\) equipped
with their geodesic metrics was posed in \cite{lim2021gromov}.  The
argument of \cite[Theorem~B]{lim2021gromov}, based on the quantitative
Borsuk--Ulam theorem of Dubins--Schwarz
\cite{dubins1981equidiscontinuity}, gives, for \(0<m<n<\infty\),
\(\dgh(\Sph^m,\Sph^n)\geq\tfrac12\zeta_m\), where
\(\zeta_m:=\arccos(-\tfrac1{m+1})\) is the common distance between
distinct vertices of a regular simplex with \(m+2\) vertices in
\(\Sph^m\).  A sharper obstruction was subsequently obtained in
\cite[Main Theorem]{adams2022gromov} by introducing
thresholds \(c_{m,n}\) derived from the equivariant topology of
\emph{Vietoris--Rips filtrations}, which we recall below; the resulting bound
\(\tfrac12c_{m,n}\) recovers, and can improve, the lower bound
\(\tfrac12\zeta_m\).  It recovers this bound because
monotonicity in the second index and
\(c_{m,m+1}=\zeta_m\) \cite[Theorem~5.2]{adams2022gromov}
give \(c_{m,n}\geq\zeta_m\) whenever
\(n\geq m+1\). We show that the coindex lower bound can underestimate
the Gromov--Hausdorff distance by a factor arbitrarily close to two, and
that this factor is optimal.

For compact metric spaces \((X,d_X)\) and \((Y,d_Y)\), a \emph{correspondence} is a relation
\(\mathcal R\subseteq X\times Y\) that projects onto both factors.
{
For a non-empty relation \(\mathcal R\subseteq X\times Y\), its
\emph{distortion} is
\[
 \dis(\mathcal R):=
 \sup\bigl\{
  |d_X(x,x')-d_Y(y,y')|:
 (x,y),(x',y')\in\mathcal R
 \bigr\}.
\]
{For a function \(f\colon X\to Y\), we write
\(\dis(f)\) for the distortion of its graph.}
}
{The \emph{correspondence formula} is
\(\dgh(X,Y)=\tfrac12\inf_{\mathcal R}\dis(\mathcal R)\), where the
infimum is over all correspondences between \(X\) and \(Y\)
\cite[Theorem~7.3.25]{burago2022course}.}
{
Taking \(\mathcal R=X\times Y\) gives
\begin{equation}
\label{eq:stream-diameter-upper-bound}
 \dgh(X,Y)\le\frac12\max\{\diam(X),\diam(Y)\}.
\end{equation}
Since round spheres have diameter \(\pi\), this gives the upper bound
\(\pi/2\) for their Gromov--Hausdorff distance.
\par}

For a metric space
\((X,d_X)\), write \(|\VR(X;r)|\) for the geometric realization of the
{\emph{Vietoris--Rips complex}} whose simplices
are the nonempty finite subsets of \(X\) of diameter at most \(r\).
Let \(C_2\) denote the group of order two, acting antipodally on each
sphere.  A function between \(C_2\)-spaces is called \emph{odd} when it
is equivariant.  For a \(C_2\)-space \(X\), its
{\emph{coindex}} is
\begin{equation}
\label{eq:coindex-intro}
 \coind(X):=
 \sup\bigl\{k\geq0:\text{there is a continuous odd map }\Sph^k\to X\bigr\}.
\end{equation}
Following \cite[Definition~1.1]{adams2022gromov}, for
\(0\leq m\leq n\) define
\begin{equation}
\label{eq:c-definition-intro}
 c_{m,n}:=
 \inf\bigl\{
 r\geq0:
 \coind\bigl(|\VR(\Sph^m;r)|\bigr)\geq n
 \bigr\}.
\end{equation}
The action on \(|\VR(\Sph^m;r)|\) is induced by
\(x\mapsto-x\) on vertices.

{
The {quantitative Borsuk--Ulam theorem} proved in
\cite[Main Theorem]{adams2022gromov} gives
\begin{equation}
\label{eq:polymath-lower-bound}
 \dgh(\Sph^m,\Sph^n)
 \geq\frac12c_{m,n}.
\end{equation}
}
{The question posed in
\cite[Question~8.1]{adams2022gromov} is whether the resulting lower
bound is always sharp:}
\begin{equation}
\label{eq:equality-question}
 \dgh(\Sph^m,\Sph^n)
 \stackrel{?}{=}
 \frac12c_{m,n}
 \qquad(0\leq m<n).
\end{equation}

{
When \(m=1\), equality holds for every \(n>1\)
\cite[Theorem~5.1]{adams2022gromov},
\cite[Theorem~1.1]{harrison2023quantitative}.
Our Main Theorem shows that, for every fixed \(m\ge3\), the inequality
is strict for all sufficiently large \(n\). Thus \(m=2\) is the only
remaining positive dimension in which the coindex lower bound could be
sharp for every higher-dimensional sphere; see
Question~\ref{ques:stream-dimension-two} below.
\par}

We obtain upper bounds for \(c_{m,n}\) from finite
configurations of points (Section~\ref{sec:stream-first-limit}), and lower
bounds for the Gromov--Hausdorff distance from obstructions to homomorphisms
between Borsuk graphs (Section~\ref{sec:stream-borsuk}). Comparing these
bounds establishes a quantitative gap without determining either quantity
exactly.

{Recall the upper bound from
\eqref{eq:stream-diameter-upper-bound}:}
\[
 \dgh(\Sph^m,\Sph^n)\le\frac\pi2
 \qquad(1\le m<n).
\]
Define the comparison ratio
\[
 R(m,n):=\frac{\dgh(\Sph^m,\Sph^n)}{\tfrac12c_{m,n}}.
\]
Since \(c_{m,n}\ge\zeta_m>\pi/2\), the bounds
\eqref{eq:polymath-lower-bound} and
\eqref{eq:stream-diameter-upper-bound} give
\[
 1\le R(m,n)<2,
\]
so \(2\) is the largest possible limiting value of \(R\).
\par

The Main Theorem below answers \eqref{eq:equality-question}
negatively in the strongest asymptotic form: \(R(m,n)\) can be
arbitrarily close to \(2\), for example when \(n=m^2\) and
\(m\to\infty\). The two quantities in this ratio are governed by
separate conditions. The
condition \(\log n/m\to0\) controls the denominator,
\(\tfrac12c_{m,n}\to\tfrac\pi4\), and the condition \(m/n\to0\)
controls the numerator, \(\dgh(\Sph^m,\Sph^n)\to\tfrac\pi2\).
Part~\ref{item:stream-main-I} of the Main Theorem bounds \(c_{m,n}\) above, part~\ref{item:stream-main-II}
bounds \(\dgh(\Sph^m,\Sph^n)\) below, and part~\ref{item:stream-main-III} compares them.

{

\begin{maintheorem}
For integers \(1\le m<n\), put
\[
 \gamma_{m,n}:=\sqrt{\frac mn}
 \left(\frac\pi2+\sqrt{\log\frac{en}{m}}\right).
\]
All limiting statements below are understood along sequences of pairs
\((m,n)\) with \(1\le m<n\).
\begin{enumerate}[label=(\Roman*),ref=(\Roman*)]
\item\label{item:stream-main-I} We have
\[
 c_{m,n}\le\pi-\arcsin(n^{-1/m}).
\]
Consequently,
\[
 \frac{\log n}{m}\longrightarrow0
 \quad\Longrightarrow\quad c_{m,n}\longrightarrow\frac\pi2.
\]
\item\label{item:stream-main-II} We have
\[
 \dgh(\Sph^m,\Sph^n)\ge\frac\pi2-\gamma_{m,n}.
\]
Consequently,
\[
 \frac mn\longrightarrow0
 \quad\Longrightarrow\quad
 \dgh(\Sph^m,\Sph^n)\longrightarrow\frac\pi2.
\]
{
For every fixed \(m\ge1\), this gives
\[
 0\le\frac\pi2-\dgh(\Sph^m,\Sph^n)
 \le\gamma_{m,n}
 =O_m\!\left(\sqrt{\frac{\log n}{n}}\right)
 \qquad(n\to\infty).
\]
In particular, it recovers the fixed-dimension limit established in
\cite[Proposition~1.8]{lim2021gromov}.
\par}
\item\label{item:stream-main-III} For every fixed \(m\ge3\), there is an integer \(N_m>m\)
such that, for every \(n\ge N_m\),
\[
 \dgh(\Sph^m,\Sph^n)-\frac12c_{m,n}
 \ge\frac14n^{-1/m}>0.
\]
\end{enumerate}
Moreover,
\[
 \frac{\log n}{m}\to0\quad\text{and}\quad\frac mn\to0
 \quad\Longrightarrow\quad R(m,n)\to2,
\]
the worst possible ratio.
\end{maintheorem}

For a concrete illustration along the family \(n=m^2\), our estimates give
\[
 \frac{\dgh(\Sph^{2026},\Sph^{2026^2})}{\tfrac12c_{2026,2026^2}}>1.737.
\]
See Remark~\ref{cor:stream-finite-example}.

{
In part~\ref{item:stream-main-III}, one may take
\[
 N_m=\left\lceil m\,40^{\,2m/(m-2)}\right\rceil;
\]
see Corollary~\ref{cor:stream-fixed-dimension-gap}.\footnote{These sufficient thresholds grow linearly with the dimension: \(N_m/m\to1600\) as \(m\to\infty\).}
For example, this choice gives
\begin{center}
 \renewcommand{\arraystretch}{1}
 \begin{tabular*}{0.35\linewidth}{@{\extracolsep{\fill}}rr@{}}
 \hline
 \(m\)&\(N_m\)\\ \hline
 3&12,288,000,000\\
 4&10,240,000\\
 5&1,094,385\\
 6&384,000\\ \hline
 \end{tabular*}\hspace{0.1\linewidth}\begin{tabular*}{0.35\linewidth}{@{\extracolsep{\fill}}rr@{}}
 \hline
 \(m\)&\(N_m\)\\ \hline
 7&214,223\\
 8&149,710\\
 9&118,530\\
 10&101,193\\ \hline
 \end{tabular*}
\end{center}

Already for \(m=3\), this gives a large sufficient threshold,
leaving open how early strict separation occurs; see
Question~\ref{ques:stream-minimal-threshold}. As recalled above, the
coindex bound is exact for \(m=1\) and every \(n>1\), whereas
part~\ref{item:stream-main-III} gives eventual strictness for every fixed
\(m\ge3\). This contrast singles out \(m=2\), for which our results do
not establish even the existence of a strict example, and motivates the
following question.

\begin{question}
\label{ques:stream-dimension-two}
Does there exist an integer \(n\ge4\) such that
\[
 \frac12c_{2,n}<\dgh(\Sph^2,\Sph^n)?
\]
If so, does this inequality hold for all sufficiently large \(n\)?
\end{question}
\par}
\par}

\paragraph{Proof strategy and organization.}
The proof combines projective packings with obstructions to homomorphisms
between Borsuk graphs. Projective packings provide finite antipodal
configurations from which continuous odd maps into Vietoris--Rips complexes
can be constructed, yielding upper bounds for \(c_{m,n}\)
(Section~\ref{sec:stream-first-limit}). To obtain lower bounds for the
Gromov--Hausdorff distance, we observe that maps of small distortion induce
homomorphisms between Borsuk graphs at suitable distance thresholds.
Such homomorphisms force points mapped into a common small cap to have
controlled diameter. Spherical isodiametry and finite weighted averaging
then impose inequalities between cap volumes. The resulting obstruction therefore applies
to maps selected from arbitrary correspondences and gives the complementary
metric lower bounds (Section~\ref{sec:stream-borsuk}).
Section~\ref{sec:stream-examples} compares these bounds to prove the optimal
factor-two loss, establish strict separation for each fixed \(m\ge3\) and
all sufficiently large \(n\), and give an explicit finite-dimensional example.

\paragraph{Acknowledgments.}
S.~Lim gratefully acknowledges support from the National Research
Foundation of Korea (NRF) through grants funded by the Korean government
(MSIT, RS-2025-23324186). F.~M\'emoli gratefully acknowledges support
from the National Science Foundation through grants CCF-2523653 and
DMS-2524362.
During the preparation of this manuscript, the authors
used OpenAI's ChatGPT for exploratory calculations and editorial
assistance.  The authors take full responsibility for the manuscript.
\section{Bounds for the Vietoris--Rips coindex threshold}
\label{sec:stream-first-limit}

In this section we bound \(c_{m,n}\) from above. A greedy volume
argument produces \(n+1\) lines in \(\mathbb R^{m+1}\) with large
pairwise angles (Lemma~\ref{lem:stream-greedy-packing}), and a
cross-polytope construction turns any such configuration into a
continuous odd map from \(\Sph^n\) into a Vietoris--Rips complex of
\(\Sph^m\) (Proposition~\ref{prop:stream-cross-polytope}). Together
they give the finite bound of
Corollary~\ref{cor:stream-finite-upper-bound}, from which part~\ref{item:stream-main-I} of
the Main Theorem follows. We begin with elementary estimates for
spherical cap volumes, which are used again in
Section~\ref{sec:stream-borsuk}.

\paragraph{Spherical cap volumes.}
Let \(\sigma_d\) be the normalized Riemannian volume measure on \(\Sph^d\), so
that \(\sigma_d(\Sph^d)=1\). For \(z\in\Sph^d\) and
\(0\le\rho\le\pi\), let \(B_{\Sph^d}(z,\rho)\) denote the open
geodesic ball (spherical cap) of radius \(\rho\) centered at \(z\),
and put
\[
 \CapVol_d(\rho):=\sigma_d(B_{\Sph^d}(z,\rho)),
 \qquad 0\le\rho\le\pi.
\]
By rotation invariance, \(\CapVol_d(\rho)\) does not depend on
\(z\). The boundary of a cap has measure zero, so the corresponding
closed cap has the same volume.

\begin{lemma}[Elementary cap estimates]
\label{lem:stream-cap-estimates}
For every integer \(d\ge1\) and \(0\le\rho\le\pi/2\),
\begin{equation}
\label{eq:stream-elementary-cap-estimates}
 \frac1{2d}\left(\frac{2\rho}{\pi}\right)^d
 \le\CapVol_d(\rho)
 \le\frac12(\sin\rho)^d.
\end{equation}
\end{lemma}

\begin{proof}
The spherical metric in polar coordinates and the volume-form formula
\cite[Sections~1.2 and~1.4.4, pp.~7 and~18--21]{petersen2016riemannian}
give
\[
 \CapVol_d(\rho)
 =\frac{\int_0^\rho\sin^{d-1}t\,dt}
        {\int_0^\pi\sin^{d-1}t\,dt}.
\]
Since \(\sin t\ge2t/\pi\) on \([0,\pi/2]\) and the denominator
is at most \(\pi\),
\[
 \CapVol_d(\rho)
 \ge\frac1\pi\left(\frac2\pi\right)^{d-1}
       \frac{\rho^d}{d}
 =\frac1{2d}\left(\frac{2\rho}{\pi}\right)^d.
\]
For the upper bound, suppose \(0<\rho<\pi/2\) and set
\(s:=\sin\rho\). Substituting \(u=\sin t\) and then \(u=sv\) gives
\begin{align*}
 \int_0^\rho\sin^{d-1}t\,dt
 &=s^d\int_0^1\frac{v^{d-1}}{\sqrt{1-s^2v^2}}\,dv\\
 &\le s^d\int_0^1\frac{v^{d-1}}{\sqrt{1-v^2}}\,dv\\
 &=s^d\int_0^{\pi/2}\sin^{d-1}t\,dt
 =\frac{s^d}{2}\int_0^\pi\sin^{d-1}t\,dt.
\end{align*}
Dividing by the denominator proves the upper bound. The endpoint cases
\(\rho=0\) and \(\rho=\pi/2\) follow by continuity.
\end{proof}

\paragraph{Greedy projective packing.}
The upper bound in Lemma~\ref{lem:stream-cap-estimates} controls
the volume excluded at each step of a greedy choice of lines.
{A real projective code is a finite collection of
lines through the origin; see \cite[Section~1]{bukh2020nearly} and
\cite[Section~2]{cohn2016optimal}. Such codes form the one-dimensional
case of Grassmannian packing; see \cite{conway1996packing}.}
For unit representatives \(x_i,x_j\in\Sph^m\), the angle between
\(\mathbb R x_i\) and \(\mathbb R x_j\) is
\[
 \arccos|\langle x_i,x_j\rangle|\in[0,\pi/2].
\]
To obtain configurations of lines with large pairwise angles, we
use the Gilbert--Varshamov volume argument in real projective space;
see Barg--Nogin
\cite[p.~2450, equation~(1)]{barg2002bounds}.  We include the short proof to obtain the explicit bound used below.

\begin{lemma}[Greedy projective packing]
\label{lem:stream-greedy-packing}
For integers \(m\ge1\) and \(N\ge2\), there exist \(N\) lines in
\(\mathbb R^{m+1}\) whose pairwise angles are at least
\[
 \theta_{m,N}:=\arcsin\!\bigl((N-1)^{-1/m}\bigr).
\]
Equivalently, there exist \(x_1,\ldots,x_N\in\Sph^m\) such that
\[
 |\langle x_i,x_j\rangle|
 \le\sqrt{1-(N-1)^{-2/m}}
 \qquad(i\ne j).
\]
\end{lemma}

\begin{proof}
Fix \(0<\theta<\theta_{m,N}\), and choose \(x_1\in\Sph^m\).
After choosing \(x_1,\ldots,x_j\), exclude the open caps of radius
\(\theta\) centered at those points and their negatives:
\[
 U_j:=\bigcup_{i=1}^j
 \bigl(B_{\Sph^m}(x_i,\theta)\cup B_{\Sph^m}(-x_i,\theta)\bigr).
\]
For \(j\le N-1\), Lemma~\ref{lem:stream-cap-estimates} gives
\[
 \sigma_m(U_j)\le2j\CapVol_m(\theta)
 \le j\sin^m\theta
 \le(N-1)\sin^m\theta<1.
\]
Thus \(\Sph^m\setminus U_j\) is nonempty, and we can choose
\(x_{j+1}\) there. Continuing gives \(N\) points such that, for
\(i\ne j\), both \(d_m(x_i,x_j)\ge\theta\) and
\(d_m(x_i,-x_j)\ge\theta\). Since
\(d_m(x_i,-x_j)=\pi-d_m(x_i,x_j)\), this means
\[
 \theta\le d_m(x_i,x_j)\le\pi-\theta,
 \qquad |\langle x_i,x_j\rangle|\le\cos\theta.
\]

To obtain the endpoint, choose \(\theta_k\uparrow\theta_{m,N}\) and
apply the construction for each \(k\). By compactness of
\((\Sph^m)^N\), a subsequence of the resulting tuples converges to
some \((x_1,\ldots,x_N)\). Continuity of the inner product gives
\[
 |\langle x_i,x_j\rangle|
 \le\cos\theta_{m,N}
 =\sqrt{1-(N-1)^{-2/m}}
 \qquad(i\ne j).
\]
Since \(\theta_{m,N}>0\), the resulting lines are distinct.
\end{proof}

\paragraph{The cross-polytope construction.}

We now convert a projective packing into a continuous odd map. The
boundary of the cross-polytope in \(\mathbb R^N\) is an antipodally
symmetric triangulation of \(\Sph^{N-1}\); sending its vertices
\(\pm e_i\) to \(\pm x_i\) and extending affinely on faces gives a map
into a Vietoris--Rips complex of \(\Sph^m\) whose scale is controlled by
the pairwise angles.
{A similar construction
appears, for Vietoris--Rips metric thickenings, in Adams--Bush--Frick
\cite[proof of Theorem~2]{adams2024projective}.}

\begin{proposition}[Cross-polytope construction]
\label{prop:stream-cross-polytope}
Let \(m\ge1\), \(N\ge2\), and \(0\le\varepsilon<1\). Suppose
\(x_1,\ldots,x_N\in\Sph^m\) satisfy
\(|\langle x_i,x_j\rangle|\le\varepsilon\) for \(i\ne j\), and put
\[
 r(\varepsilon):=\arccos(-\varepsilon)=\frac\pi2+\arcsin\varepsilon.
\]
Then there is a continuous odd map
\[
 \Sph^{N-1}\longrightarrow |\VR(\Sph^m;r(\varepsilon))|.
\]
{In particular, if \(N\ge m+1\), so that \(c_{m,N-1}\) is
defined,} then
\(c_{m,N-1}\le r(\varepsilon)\).
\end{proposition}

\begin{proof}
For any signs \(\sigma_i,\sigma_j\in\{-1,1\}\) and distinct indices,
we have \(\langle\sigma_ix_i,\sigma_jx_j\rangle\ge-\varepsilon\),
and hence
\[
 d_m(\sigma_i x_i,\sigma_j x_j)=\arccos\langle\sigma_ix_i,\sigma_jx_j\rangle\le r(\varepsilon).
\]
Consequently, every set obtained by choosing one point from each pair
\(\{x_i,-x_i\}\) has diameter at most \(r(\varepsilon)\).

Consider the cross-polytope
\[
 {\mathcal C_N}:=\operatorname{conv}\{\pm e_1,\ldots,\pm e_N\}
 \subset\mathbb R^N.
\]
Its boundary is equivariantly homeomorphic to
\(\Sph^{N-1}\), and each maximal boundary face has vertices
\(\sigma_1e_1,\ldots,\sigma_Ne_N\), with one sign per index.
The assignment \(\sigma_i e_i\mapsto\sigma_i x_i\) sends each face
to a Vietoris--Rips simplex at scale \(r(\varepsilon)\). Extend by preserving barycentric
coordinates:
\[
 \sum_i\lambda_i\sigma_i e_i
 \longmapsto
 \sum_i\lambda_i[\sigma_i x_i],
 \qquad \lambda_i\ge0,\quad\sum_i\lambda_i=1.
\]
The sum on the right is a point of the geometric realization of the
Vietoris--Rips complex. These affine maps agree on shared faces and commute with
negation. They therefore define the required continuous odd map.
{If \(N\ge m+1\), the inequality
\(c_{m,N-1}\le r(\varepsilon)\) follows from the definition of
\(c_{m,N-1}\).}
\end{proof}

{Applying Proposition~\ref{prop:stream-cross-polytope} to the greedy
packing of Lemma~\ref{lem:stream-greedy-packing} with \(N=n+1\) lines
gives an explicit upper bound for \(c_{m,n}\).
\begin{corollary}[Finite upper bound]
\label{cor:stream-finite-upper-bound}
For integers \(1\le m<n\),
\[
 c_{m,n}\le\pi-\arcsin(n^{-1/m}).
\]
\end{corollary}

\begin{proof}
Apply Lemma~\ref{lem:stream-greedy-packing} with \(N=n+1\), and set
\(\varepsilon:=\cos\theta_{m,n+1}\in[0,1)\).
Proposition~\ref{prop:stream-cross-polytope} gives
\[
 c_{m,n}\le r(\varepsilon)
 =\pi-\theta_{m,n+1}
 =\pi-\arcsin(n^{-1/m}).
\]
\end{proof}

\paragraph{Proof of the limit in part \ref{item:stream-main-I} of the Main Theorem.}
We combine Corollary~\ref{cor:stream-finite-upper-bound} with the
lower bound \(c_{m,n}\ge\zeta_m\) recalled in the introduction.
\begin{proof}
The hypothesis \(\log n/m\to0\) forces \(m\to\infty\) and gives
\[
 n^{-1/m}
 =\exp\!\left(-\frac{\log n}{m}\right)\longrightarrow1.
\]
Corollary~\ref{cor:stream-finite-upper-bound} and the lower
bound \(c_{m,n}\ge\zeta_m\) therefore yield
\[
 \frac\pi2<\arccos\!\left(-\frac1{m+1}\right)
 \le c_{m,n}
 \le\pi-\arcsin(n^{-1/m})
 \longrightarrow\frac\pi2.
\]
\end{proof}
\par}

\section{Borsuk graphs and the Gromov--Hausdorff distance}
\label{sec:stream-borsuk}
In this section we bound \(\dgh(\Sph^m,\Sph^n)\) from below.
We first show that maps of small distortion induce homomorphisms
between Borsuk graphs, and that such homomorphisms must respect
normalized cap volumes (Proposition~\ref{prop:stream-cap}). Inverting
this cap-volume inequality gives a lower bound for every pair
\(1\le m<n\) (Corollary~\ref{cor:stream-inverse-cap}). The elementary
cap estimates of Lemma~\ref{lem:stream-cap-estimates} then make this
bound explicit (Corollary~\ref{cor:stream-explicit-gh}), which proves
part~\ref{item:stream-main-II} of the Main Theorem, including the case of fixed \(m\).

\paragraph{The Borsuk graph filtration and its stability.}

We begin by placing Borsuk graphs in the graph-homomorphism
preorder and comparing the resulting filtrations under low-distortion
maps.

\begin{definition}[Borsuk graphs]
\label{def:stream-borsuk}
For graphs \(G,H\), write \(G\preceq H\) if there is a graph homomorphism
\(G\to H\). We regard \(\preceq\) as a preorder on graphs.

For a compact metric space \((X,d_X)\) and \(t>0\), the \emph{closed
Borsuk graph} \(\Bor_t(X)\) is the simple graph with vertex set
\(X\) in which distinct points \(x,x'\) are adjacent precisely when
\(d_X(x,x')\ge t\). If \(0<t<s\), the identity on \(X\) induces a graph
homomorphism
\[
 \Bor_s(X)\longrightarrow\Bor_t(X).
\]
We denote this decreasing filtration in the graph-homomorphism preorder
by \(\Bor_\bullet(X)\).
\end{definition}

For a general metric space, the analogous graph with the strict
threshold \(d_X(x,x')>s\) is the
\(1\)-skeleton of Engstr\"om's anti-Rips complex
\cite[Definition~4.1]{engstrom2009complexes}. On a round sphere,
after the angular--chordal conversion
\(\|x-x'\|=2\sin(d_m(x,x')/2)\), \(\Bor_s(\Sph^m)\) is the
classical Borsuk graph; see \cite[p.~317]{lovasz1983selfdual}. Kahle
and Martinez-Figueroa use this spherical formulation and study its
random induced subgraphs in \cite[Section~1]{kahle2020random-borsuk}.

To compare the Borsuk-graph filtrations of two spaces, we use a notion of distance that  is modeled on the interleaving distance used in
persistence theory; see \cite[Section~3]{bubenik2015metrics}.
\begin{definition}[Interleaving distance]
\label{def:stream-interleaving}
{For compact metric spaces \((X,d_X)\) and \((Y,d_Y)\), define
\(\dI\bigl(\Bor_\bullet(X),\Bor_\bullet(Y)\bigr)\) as the infimum of shifts
\(\delta\ge0\) such that, for every \(u>\delta\), graph homomorphisms exist
in both directions:}
\[
 \Bor_u(X)\longrightarrow\Bor_{u-\delta}(Y),
 \qquad
 \Bor_u(Y)\longrightarrow\Bor_{u-\delta}(X).
\]
Since a preorder has at most one morphism between any two
objects, the compatibility diagrams for an
interleaving commute automatically. Interleavings compose, so
\(\dI\) is an extended pseudometric. The maps are not required to be
continuous or measurable.
\end{definition}

The restriction \(u>\delta\) ensures that \(u-\delta>0\).

\begin{proposition}[Stability]
\label{prop:stream-stability}
For compact metric spaces \((X,d_X)\) and \((Y,d_Y)\),
\[
 {\dI\bigl(\Bor_\bullet(X),\Bor_\bullet(Y)\bigr)\le2\dgh(X,Y).}
\]
\end{proposition}

\begin{proof}
A map \(f\colon X\to Y\) of distortion at most \(\delta<s\) induces
\(\Bor_s(X)\to\Bor_{s-\delta}(Y)\): if
\(d_X(x,x')\ge s\), then \(d_Y(f(x),f(x'))\ge s-\delta>0\), so
\(f(x)\) and \(f(x')\) are distinct and adjacent.
For any \(\delta>2\dgh(X,Y)\), choose a correspondence between
\(X\) and \(Y\) of distortion less than \(\delta\). Selecting maps
in both directions from this correspondence gives maps of distortion
less than \(\delta\), hence a \(\delta\)-interleaving. Thus
\({\dI\bigl(\Bor_\bullet(X),\Bor_\bullet(Y)\bigr)}\le\delta\);
let \(\delta\downarrow2\dgh(X,Y)\).
\end{proof}

\paragraph{Cap-volume monotonicity under graph homomorphisms.}

Spherical isodiametry states that a spherical cap
has maximal volume among measurable sets of a prescribed diameter less
than \(\pi\). We use it to prove the cap obstruction in
Proposition~\ref{prop:stream-cap}, which yields the lower bound needed
for part~\ref{item:stream-main-II} of the Main Theorem.\footnote{We recently learned that
Mallick~\cite{mallick2026borsuk} announced related bounds for measurable
homomorphisms between Borsuk graphs, using independence ratios. Our proof
of Proposition~\ref{prop:stream-cap} was developed independently and
requires no measurability assumption.}
\begin{theorem}[Spherical isodiametry {\cite[Theorem~1.2]{boroczky2019isodiametric}}]
\label{thm:stream-isodiametry}
Let \(n\ge2\) and \(0<D<\pi\). Every measurable set
\(E\subseteq\Sph^n\) of diameter at most \(D\), measured with \(d_n\),
satisfies
\[
 \sigma_n(E)\le\CapVol_n(D/2).
\]
\end{theorem}

\begin{proposition}[The cap-volume obstruction]
\label{prop:stream-cap}
For \(n\ge2\), \(m\ge1\), and \(0<s,t<\pi\), an arbitrary graph
homomorphism
\[
 f:\Bor_s(\Sph^n)\longrightarrow\Bor_t(\Sph^m)
\]
forces \(\CapVol_n(s/2)\ge\CapVol_m(t/2)\).
\end{proposition}

\begin{proof}
Fix \(\varepsilon>0\) with \(s+2\varepsilon<\pi\). Choose a finite
\(\varepsilon\)-net \(x_1,\ldots,x_N\) in the source \(\Sph^n\), so
that every point is at distance strictly less than \(\varepsilon\)
from some \(x_j\). Its closed Voronoi regions are
\[
 V_j:=\{x\in\Sph^n:d_n(x,x_j)\le d_n(x,x_k)
                 \text{ for all }k\}.
\]
These regions cover \(\Sph^n\), and
\(V_j\subseteq B_{\Sph^n}(x_j,\varepsilon)\). Disjointify them in
index order:
\[
 E_1:=V_1,\qquad E_j:=V_j\setminus\bigcup_{i<j}V_i
 \quad(2\le j\le N).
\]
The sets \(E_j\) form a Borel partition of \(\Sph^n\), with
\(E_j\subseteq V_j\). Set \(w_j:=\sigma_n(E_j)\), so
\(\sum_jw_j=1\). For each \(z\in\Sph^m\), set
\[
 J_z:=\{j:f(x_j)\in B_{\Sph^m}(z,t/2)\},
 \qquad U_z:=\bigcup_{j\in J_z}E_j,
\]
using open target caps.
If \(J_z=\varnothing\), then \(U_z=\varnothing\) and the bound
\(\sigma_n(U_z)\le\CapVol_n(s/2+\varepsilon)\) is immediate.
Otherwise, if \(j,k\in J_z\), then
\(d_m(f(x_j),f(x_k))<t\), so \(f(x_j)\) and \(f(x_k)\) are not
adjacent in \(\Bor_t(\Sph^m)\). Since \(f\) is a graph
homomorphism, \(x_j\) and \(x_k\) are not adjacent in
\(\Bor_s(\Sph^n)\) either, so \(d_n(x_j,x_k)<s\). For \(x\in E_j\) and \(y\in E_k\), it follows
that
\[
 d_n(x,y)
 \le d_n(x,x_j)+d_n(x_j,x_k)+d_n(x_k,y)
 <s+2\varepsilon.
\]
Thus \(\diam(U_z)\le s+2\varepsilon\). Each \(U_z\) is a finite
union of Borel sets, so Theorem~\ref{thm:stream-isodiametry} gives
\(\sigma_n(U_z)\le\CapVol_n(s/2+\varepsilon)\).

By symmetry of the distance, a fixed point \(f(x_j)\) belongs to
\(B_{\Sph^m}(z,t/2)\) exactly when
\(z\in B_{\Sph^m}(f(x_j),t/2)\). The set of such centers therefore
has \(\sigma_m\)-measure \(\CapVol_m(t/2)\). Since the cells are
disjoint, the function
\[
 F(z):=\sigma_n(U_z)
 =\sum_{j=1}^N w_j\,
   \mathbf 1_{B_{\Sph^m}(f(x_j),t/2)}(z)
\]
is a finite sum of measurable indicator functions. Consequently,
\begin{align*}
 \Av_{\sigma_m}(F)
 &:=\int_{\Sph^m}F(z)\,d\sigma_m(z)\\
 &=\sum_{j=1}^N w_j\,
   \sigma_m\!\left(B_{\Sph^m}(f(x_j),t/2)\right)\\
 &=\sum_{j=1}^N w_j\CapVol_m(t/2)
 =\CapVol_m(t/2).
\end{align*}
Only the finitely many fixed values \(f(x_j)\) enter this measurable
function of the center \(z\); no measurability of \(f\) is needed.
Averaging the bound \(F(z)\le\CapVol_n(s/2+\varepsilon)\),
obtained above for every \(z\in\Sph^m\), yields
\[
 \CapVol_m(t/2)=\Av_{\sigma_m}(F)
 \le\CapVol_n(s/2+\varepsilon).
\]
Letting \(\varepsilon\downarrow0\) and using continuity of
\(\CapVol_n\) proves the assertion.
\end{proof}

\paragraph{A Gromov--Hausdorff lower bound.}

Inverting the cap-volume inequality gives a lower
bound for the interleaving distance between the Borsuk-graph filtrations
of \(\Sph^m\) and \(\Sph^n\), and stability transfers this bound
to the Gromov--Hausdorff distance.
\begin{corollary}[Effective lower bound]
\label{cor:stream-inverse-cap}
For \(1\le m<n\),
{
\[
\begin{aligned}
 \dgh(\Sph^m,\Sph^n)
 &\ge\frac12\dI\bigl(\Bor_\bullet(\Sph^m),\Bor_\bullet(\Sph^n)\bigr)\\
 &\ge\sup_{0<\rho<\pi/2}
 \left[\CapVol_n^{-1}\!\bigl(\CapVol_m(\rho)\bigr)-\rho\right].
\end{aligned}
\]
\par}
\end{corollary}

\begin{proof}
Fix \(\rho\in(0,\pi/2)\), and write
\(a:=\CapVol_n^{-1}(\CapVol_m(\rho))\in(0,\pi/2)\).
If \(a\le\rho\), then
\[
 {\dI\bigl(\Bor_\bullet(\Sph^m),\Bor_\bullet(\Sph^n)\bigr)}\ge2(a-\rho)
\]
is immediate. Otherwise, let \(0\le\delta<2(a-\rho)\).
A \(\delta\)-interleaving would give a homomorphism
\[
 \Bor_{2\rho+\delta}(\Sph^n)
 \longrightarrow\Bor_{2\rho}(\Sph^m).
\]
Since \(0<2\rho+\delta<2a<\pi\),
Proposition~\ref{prop:stream-cap} would imply
\[
 \CapVol_n(\rho+\delta/2)
 \ge\CapVol_m(\rho)=\CapVol_n(a),
\]
contradicting strict monotonicity of \(\CapVol_n\), because
\(\rho+\delta/2<a\). Thus no such \(\delta\)-interleaving exists,
and
\[
 {\dI\bigl(\Bor_\bullet(\Sph^m),\Bor_\bullet(\Sph^n)\bigr)}\ge2(a-\rho).
\]
Taking the supremum
over \(\rho\) and applying Proposition~\ref{prop:stream-stability}
proves the claimed inequalities.
\end{proof}

Colding's volume estimate \cite[Lemma~5.10]{colding1996large}, in
the sphere form of \cite[Proposition~1.2]{lim2021gromov}, also bounds
\(\dgh(\Sph^m,\Sph^n)\) below in terms of cap volumes. The two bounds compare as follows.

\begin{remark}[Comparison with Colding's volume bound]
For fixed \(1\le m<n\) and \(0<\rho<\pi/2\), define
\begin{align*}
 \ell_B(\rho)&:=\CapVol_n^{-1}\!\bigl(\CapVol_m(\rho)\bigr)-\rho,\\
 \ell_C(\rho)&:=\frac12\left[
 \CapVol_n^{-1}\!\bigl(\CapVol_m(\rho/2)\bigr)-\rho\right].
\end{align*}
Corollary~\ref{cor:stream-inverse-cap} gives
{
\[
\begin{aligned}
 \dgh(\Sph^m,\Sph^n)
 &\ge\frac12\dI\bigl(\Bor_\bullet(\Sph^m),\Bor_\bullet(\Sph^n)\bigr)\\
 &\ge\sup_{0<\rho<\pi/2}\ell_B(\rho),
\end{aligned}
\]
\par}
whereas Colding's estimate gives
\begin{equation}
\label{eq:stream-colding}
 \dgh(\Sph^m,\Sph^n)\ge\sup_{0<\rho<\pi/2}\ell_C(\rho).
\end{equation}
The original Colding bound allows \(0<\rho\le\pi\). Radii
\(\rho\ge\pi/2\) contribute nonpositive values because
\(\CapVol_m(\rho/2)\le1/2\), while \(\ell_C(\rho)\to0\) as
\(\rho\downarrow0\); hence restricting the supremum loses nothing.
Strict monotonicity of the cap-volume functions gives
\[
 \ell_B(\rho)-2\ell_C(\rho)
 =\CapVol_n^{-1}\!\bigl(\CapVol_m(\rho)\bigr)
  -\CapVol_n^{-1}\!\bigl(\CapVol_m(\rho/2)\bigr)>0.
\]
Consequently,
\[
 \sup_{0<\rho<\pi/2}\ell_B(\rho)
 \ge2\sup_{0<\rho<\pi/2}\ell_C(\rho).
\]
Thus the lower bound in Corollary~\ref{cor:stream-inverse-cap} is at
least twice the lower bound in \eqref{eq:stream-colding}.
\end{remark}

{
\paragraph{An explicit lower bound.}
\label{sec:stream-second-limit}
The elementary cap estimates in Lemma~\ref{lem:stream-cap-estimates}
give an explicit consequence of Corollary~\ref{cor:stream-inverse-cap}.

\begin{corollary}[Explicit lower bound]
\label{cor:stream-explicit-gh}
For integers \(1\le m<n\), put
\[
 \gamma_{m,n}:=\sqrt{\frac mn}
 \left(\frac\pi2+\sqrt{\log\frac{en}{m}}\right).
\]
Then
\[
 \dgh(\Sph^m,\Sph^n)\ge\frac\pi2-\gamma_{m,n}.
\]
\end{corollary}

\begin{proof}
We bound the cosine by a Gaussian: for \(0\le t<\pi/2\),
\[
 \log\cos t=-\int_0^t\tan u\,du\le-\frac{t^2}{2},
\]
since \(\tan u\ge u\).
The two cap radii below lie in \((0,\pi/2)\), since
\(0<m/n<1\) and \(\frac mn\log\frac{en}{m}<1\); the latter holds
because \(q\mapsto q\log(e/q)\) is increasing on \((0,1]\) and equals
\(1\) at \(q=1\).
Lemma~\ref{lem:stream-cap-estimates} therefore gives
\begin{align*}
 \CapVol_n\!\left(\frac\pi2-\sqrt{\frac mn\log\frac{en}{m}}\right)
 &\le\frac12\cos^n\!\left(\sqrt{\frac mn\log\frac{en}{m}}\right)\\
 &\le\frac12\exp\!\left(-\frac m2\log\frac{en}{m}\right)\\
 &=\frac12\left(\frac mn\right)^{m/2}e^{-m/2}\\
 &\le\frac1{2m}\left(\frac mn\right)^{m/2}
 \le\CapVol_m\!\left(\frac\pi2\sqrt{\frac mn}\right).
\end{align*}
The penultimate inequality uses \(\log m\le m/2\), valid for every
integer \(m\ge1\). Applying Corollary~\ref{cor:stream-inverse-cap}
at radius \(\frac\pi2\sqrt{m/n}\) and using strict monotonicity of
\(\CapVol_n\) yields
\[
 \dgh(\Sph^m,\Sph^n)
 \ge\frac\pi2-\sqrt{\frac mn\log\frac{en}{m}}
      -\frac\pi2\sqrt{\frac mn}
 =\frac\pi2-\gamma_{m,n}.
\]
\end{proof}

\paragraph{Proof of the limit in part \ref{item:stream-main-II} of the Main Theorem.}

\begin{proof}
Along any sequence of pairs \(1\le m<n\) with \(m/n\to0\),
\[
 \gamma_{m,n}
 =\frac\pi2\sqrt{\frac mn}+\sqrt{\frac mn\log\frac{en}{m}}
 \longrightarrow0.
\]
Corollary~\ref{cor:stream-explicit-gh} and the upper bound
\(\dgh(\Sph^m,\Sph^n)\le\pi/2\) prove the claimed limit.
\end{proof}

The quantitative estimate above will also yield the strict separation
in part~\ref{item:stream-main-III}.
\par}

\section{Consequences and the finite example}
\label{sec:stream-examples}
We now combine the upper bounds for \(c_{m,n}\) from
Section~\ref{sec:stream-first-limit} with the lower bounds for
\(\dgh(\Sph^m,\Sph^n)\) from Section~\ref{sec:stream-borsuk}. This will
prove the ratio statement and part~\ref{item:stream-main-III} of the Main Theorem, and gives
an explicit finite example.

{

\paragraph{The comparison ratio.}
If \(\log n/m\to0\) and \(m/n\to0\), the limits in parts~\ref{item:stream-main-I}
and~\ref{item:stream-main-II} of the Main Theorem give
\[
 R(m,n)\longrightarrow\frac{\pi/2}{\pi/4}=2.
\]
This proves the final assertion of the Main Theorem.

{
For every fixed \(m\ge3\),
Corollaries~\ref{cor:stream-explicit-gh}
and~\ref{cor:stream-finite-upper-bound} give, as \(n\to\infty\),
\begin{align*}
 \frac\pi2-\dgh(\Sph^m,\Sph^n)&=o(n^{-1/m}),\\
 \frac\pi2-\frac12c_{m,n}&\ge\frac12n^{-1/m}.
\end{align*}
Thus the Gromov--Hausdorff distance approaches \(\pi/2\) with a deficit
negligible compared with the deficit of the coindex lower bound.
Corollary~\ref{cor:stream-fixed-dimension-gap} makes the resulting
separation explicit.
\par}

\begin{corollary}[Separation in fixed dimension]
\label{cor:stream-fixed-dimension-gap}
For integers \(1\le m<n\),
\[
 \dgh(\Sph^m,\Sph^n)-\frac12c_{m,n}
 \ge\frac12\arcsin(n^{-1/m})-\gamma_{m,n}.
\]
In particular, for every integer \(m\ge3\), put
\[
 N_m:=\left\lceil m\,40^{\,2m/(m-2)}\right\rceil.
\]
Then \(N_m>m\), and for all \(n\ge N_m\),
\[
 \dgh(\Sph^m,\Sph^n)-\frac12c_{m,n}
 \ge\frac14n^{-1/m}>0.
\]
\end{corollary}

\begin{proof}
Subtract half the upper bound in
Corollary~\ref{cor:stream-finite-upper-bound} from the lower bound in
Corollary~\ref{cor:stream-explicit-gh} to obtain the first inequality.

For \(m\ge3\) and \(n\ge N_m\), put
\(\tau:=(n/m)^{(m-2)/(2m)}\), so \(\tau\ge40\).
Using \(m^{1/m}\le3/2\) and \(2m/(m-2)\le6\), we obtain
\begin{align*}
 n^{1/m}\gamma_{m,n}
 &=\frac{m^{1/m}}{\tau}
   \left(\frac\pi2+\sqrt{1+\frac{2m}{m-2}\log\tau}\right)\\
 &\le\frac3{2\tau}\left(\frac\pi2+\sqrt{1+6\log\tau}\right)\\
 &\le\frac3{80}\left(\frac\pi2+\sqrt{1+6\log40}\right)\\
 &<\frac3{80}\left(\frac\pi2+5\right)<\frac14.
\end{align*}
The second inequality uses that
\(\tau\mapsto\tau^{-1}(\pi/2+\sqrt{1+6\log\tau})\) is decreasing for
\(\tau\ge40\); the last line uses \(\log40<4\) and \(\pi/2<5/3\).
Thus \(\gamma_{m,n}\le\tfrac14n^{-1/m}\). Since \(\arcsin t\ge t\)
on \([0,1]\),
\[
 \frac12\arcsin(n^{-1/m})-\gamma_{m,n}
 \ge\left(\frac12-\frac14\right)n^{-1/m}
 =\frac14n^{-1/m}>0.
\]
\end{proof}

{
For \(m\ge3\), let \(N_m^{\min}\) be the smallest integer \(N>m\) such that
\[
 \frac12c_{m,n}<\dgh(\Sph^m,\Sph^n)
 \qquad\text{for every integer }n\ge N.
\]
This integer exists by Corollary~\ref{cor:stream-fixed-dimension-gap},
which gives \(N_m^{\min}\le N_m\).

\begin{question}
\label{ques:stream-minimal-threshold}
For each \(m\ge3\), determine \(N_m^{\min}\).
In particular, what is \(N_3^{\min}\)?
\end{question}
\par}

\begin{remark}
\label{rem:stream-dimension-two}
The restriction \(m\ge3\) in
Corollary~\ref{cor:stream-fixed-dimension-gap} comes from comparing the
two terms of its first inequality. For fixed \(m\), as \(n\to\infty\),
the projective-packing term satisfies
\[\frac12\arcsin(n^{-1/m})\sim\frac12n^{-1/m},\]
while the error
\(\gamma_{m,n}\) in Corollary~\ref{cor:stream-explicit-gh} is of order
\(\sqrt{\log n/n}\). The first term dominates when \(m\ge3\), but for
\(m=2\),
\[
 \frac12\arcsin(n^{-1/2})=o(\gamma_{2,n}).
\]
{
Our estimates therefore do not resolve
Question~\ref{ques:stream-dimension-two}.
\par}
\end{remark}

\begin{remark}[A numerical illustration]
\label{cor:stream-finite-example}
Taking \(m=2026\) and \(n=2026^2\) in
Corollaries~\ref{cor:stream-finite-upper-bound}
and~\ref{cor:stream-explicit-gh} gives
\[
 \dgh(\Sph^{2026},\Sph^{2026^2})>1.47069,
 \qquad \frac12c_{2026,2026^2}<0.84663.
\]
Consequently,
\[
 R(2026,2026^2)>1.737,\qquad
 \dgh(\Sph^{2026},\Sph^{2026^2})-\frac12c_{2026,2026^2}>0.624.
\]
\end{remark}
\par}

\end{document}